\documentclass[a4paper]{article}
\usepackage[utf8]{inputenc}

\usepackage{hyperref}

\usepackage{tikz}
\usetikzlibrary{calc}
\usepackage{amsthm}
\usepackage{amsopn}
\usepackage{amsmath}
\usepackage{amssymb}
\usepackage{mathrsfs}
\usepackage{todonotes}

\usepackage{color}

\theoremstyle{plain}
\newtheorem{thm}{Theorem}
\newtheorem{cor}[thm]{Corollary}
\newtheorem{lemma}[thm]{Lemma}
\newtheorem{prop}[thm]{Proposition}
\newtheorem{conjecture}[thm]{Conjecture}

\theoremstyle{definition}

\newtheorem{claim}{Claim}
\newcommand\resetClaimCounter{\setcounter{claim}{0}}

\newcommand*{\myproofname}{Proof}
\newenvironment{claimproof}[1][\myproofname]{\begin{proof}[#1]}{\end{proof}}

\newcommand{\sst}[2]{\left\{\;#1 \;\middle|\; #2 \;\right\}}

\DeclareMathOperator{\Arb}{Arb}
\DeclareMathOperator{\supp}{supp}

\title{Strong chromatic index and Hadwiger number}
\author{
Wouter Cames van Batenburg
\thanks{Department of Computer Science, Universit\'e Libre de Bruxelles, Belgium. 
Email: \protect\href{mailto:wcamesva@ulb.ac.be}{\protect\nolinkurl{wcamesva@ulb.ac.be}}. Supported by the Netherlands Organisation for Scientific Research (NWO), grant no.~613.001.217, and an ARC grant from the Wallonia-Brussels Federation of Belgium.}
\and
R\'emi de Joannis de Verclos
\thanks{Department of Mathematics, Radboud University Nijmegen, Netherlands. 
Email: \protect\href{mailto:r.deverclos@math.ru.nl}{\protect\nolinkurl{r.deverclos@math.ru.nl}}. Supported by a Vidi grant (639.032.614) of the Netherlands Organisation for Scientific Research (NWO).}
\and
Ross J. Kang
\thanks{Department of Mathematics, Radboud University Nijmegen, Netherlands. 
Email: \protect\href{mailto:ross.kang@gmail.com}{\protect\nolinkurl{ross.kang@gmail.com}}. Supported by a Vidi grant (639.032.614) of the Netherlands Organisation for Scientific Research (NWO).}
\and
Fran\c{c}ois Pirot
\thanks{Department of Mathematics, Radboud University Nijmegen, Netherlands
and LORIA, Universit\'e de Lorraine, Nancy, France.
Email: \protect\href{mailto:francois.pirot@loria.fr}{\protect\nolinkurl{francois.pirot@loria.fr}}.}
}
\date{\today}

\begin{document}

\maketitle

\begin{abstract}
We investigate the effect of a fixed forbidden clique minor upon the strong chromatic index, both in multigraphs and in simple graphs.

We conjecture for each $k\ge 4$ that any $K_k$-minor-free multigraph of maximum degree $\Delta$ has strong chromatic index at most $\frac32(k-2)\Delta$. We present a construction certifying that if true the conjecture is asymptotically sharp as $\Delta\to\infty$.
In support of the conjecture, we show it in the case $k=4$ and prove the statement for strong clique number in place of strong chromatic index.

By contrast, we make a basic observation that for $K_k$-minor-free simple graphs, the problem of strong edge-colouring is ``between'' Hadwiger's Conjecture and its fractional relaxation.

For $k\geq5$, we also show that $K_k$-minor-free multigraphs of edge-diameter at most $2$ have strong clique number at most $(k-\frac{1}{2})\Delta$.
\end{abstract}

\section{Introduction}\label{sec:intro}

All multigraphs in this paper are loopless. A {\em strong edge-colouring} of multigraph $G=(V,E)$ is a partition of the edges $E$ into parts each of which is an induced matching in $G$; the {\em strong chromatic index} $\chi_2'(G)$ of $G$ is the least number of parts in such a partition. 
Notice that $\chi_2'(G)$ is equal to $\chi(L(G)^2)$, the chromatic number of the square of the line graph of $G$. 
Here the \emph{square} $H^2$ of a graph $H$ is obtained by adding an edge between any two vertices that are at distance two in $H$.
Although simply defined, $\chi_2'$ has proven difficult to pin down: a conjecture of Erd\H{o}s and Ne\v{s}et\v{r}il~\cite{Erd88} from the 1980s about the extremal behaviour of $\chi_2'$ in general is notorious; see also~\cite{CaKa19,JKP19} for its context with respect to multigraphs.

It is eminently natural to focus on graph classes that are closed under taking minors.
Already in 1990, Faudree, Gy\'arf\'as, Schelp, and Tuza~\cite{FSGT90}, by cutely combining the Four Colour Theorem with Vizing's Theorem, determined the asymptotic behaviour of $\chi_2'$ for planar graphs.

\begin{thm}[\cite{FSGT90}]\label{thm:FSGT}
For any planar graph $G$ of maximum degree $\Delta$, the strong chromatic index of $G$ satisfies $\chi_2'(G) \le 4\Delta+4$.
There is a planar graph $G$ of maximum degree $\Delta$ satisfying $\chi_2'(G) = 4\Delta-4$.
\end{thm}

\noindent
It turns out that Theorem~\ref{thm:FSGT} belongs to a broader basic phenomenon involving fractional colouring. 
 Writing ${\cal I}(G)$ for the collection of stable sets in $G$, the \emph{total weight} of a \emph{weight function} $w: {\cal I}(G)\to {\mathbb R}_0^+$ on the stable sets of $G$ is defined as the sum of the weights, $\sum_{S \in {\cal I}(G)} w(S)$. 
If the total weight is at most some nonnegative real number $r$ and additionally $\sum_{S \ni v, S\in {\cal I}(G)} w(S) \ge 1$ holds for every vertex $v$, then we say that $w$ is a \emph{fractional $r$-colouring} of $G$.
 The {\em fractional chromatic number $\chi_f(G)$} of $G$ is the smallest $r\in \mathbb{R}_{0}^{+}$ such that $G$ has a fractional $r$-colouring.
 Fractional chromatic number is the linear programming relaxation of chromatic number, and as such it is widely studied. 
In particular $\chi_f(G)\leq \chi(G)$ holds for every graph $G$.

For any class ${\cal G}$ of graphs, let us write $\chi({\cal G})$ (respectively $\chi_f({\cal G})$) for the largest chromatic number (respectively fractional chromatic number) over the graphs in $\cal G$, and $\chi_2'({\cal G},\Delta)$ for the largest strong chromatic index over the graphs in $\cal G$ of maximum degree $\Delta$.

\begin{thm}\label{thm:reduction}
Let ${\cal G}$ be a class of graphs that is closed under contraction of edges and under the attachment of pendant edges.  Then
\[
\chi_f({\cal G}) \le \limsup_{\Delta\to\infty} \frac{\chi_2'({\cal G},\Delta)}{\Delta} \le \chi({\cal G}).
\]
\end{thm}

\noindent
This generalisation of Theorem~\ref{thm:FSGT} closely links the optimisation of strong chromatic index with that of chromatic number in minor-closed graph classes.
In particular, 
the problem of finding the asymptotically best strong edge-colouring bounds for graphs without some fixed clique minor is a problem intermediary to Hadwiger's Conjecture~\cite{Had43} and its LP relaxation. Essentially via Theorem~\ref{thm:reduction}, the following would hold if Hadwiger's Conjecture holds, and moreover the conjecture would imply a fractional form of Hadwiger's Conjecture if true.
\begin{conjecture}\label{conj:intermediate}
For any $K_k$-minor-free graph $G$ of maximum degree $\Delta$, the strong chromatic index of $G$ satisfies $\chi_2'(G) \le (k-\frac12) \Delta$.  
\end{conjecture}
\noindent
Curiously the bound in Conjecture~\ref{conj:intermediate} is sharp for a \emph{blown-up $5$-cycle}, the same example conjectured to be extremal for the Erd\H{o}s-Ne\v{s}et\v{r}il Conjecture; we define and discuss this and other constructions below Theorem~\ref{thm:isclique}.
Note that the best general upper bounds on the fractional chromatic and chromatic numbers in terms of Hadwiger number have long been due respectively to Reed and Seymour~\cite{ReSe98} and to Kostochka~\cite{Kos82} (but there have been subsequent constant factor improvements). 
After submission of this manuscript, Norin, Song and Postle~\cite{NPS20,P20} made significant further progress for the chromatic number.

While the discussion thus far concerned simple graphs, mysteriously the situation seems to be quite different for multigraphs.
In particular, Van Loon and the third author~\cite{LoKa19} recently noticed a tantalising gap with respect to the strong chromatic index of planar {\em multi}graphs.
Drawing inspiration from~\cite{FSGT90}, the upper bound below is a combination of the Four Colour Theorem with Shannon's Theorem (instead of Vizing's Theorem); the lower bound construction is a modified octahedral graph.

\begin{thm}[\cite{LoKa19}]\label{thm:LoKa}
For any planar multigraph $G$ of maximum degree $\Delta$, the strong chromatic index of $G$ satisfies $\chi_2'(G) \le 6\Delta$.
There is a planar multigraph $G$ of maximum degree $\Delta$ satisfying $\chi_2'(G) = \frac92\Delta-3$.
\end{thm}

\noindent
It was also conjectured~\cite[Conj.~3.2]{LoKa19} that the construction in Theorem~\ref{thm:LoKa} has the correct asymptotically extremal behaviour. In terms of Hadwiger number, we go even further by proposing the following.

\begin{conjecture}\label{conj:main}
Let $k\ge 4$.
For any $K_k$-minor-free multigraph $G$ of maximum degree $\Delta$, the strong chromatic index of $G$ satisfies $\chi_2'(G) \le \frac32(k-2) \Delta$.
\end{conjecture}

\noindent
In Section~\ref{sec:diabolical}, we give an elementary construction to certify that the bound in Conjecture~\ref{conj:main} is asymptotically sharp if true, for each fixed $k\ge 4$.
For $k=3$ the bound fails by considering a single edge to each endpoint of which is attached $\Delta-1$ pendant edges.
In Section~\ref{sec:K4}, we prove the conjectured bound in the case $k=4$.
\begin{thm}\label{thm:K4}
  For any $K_4$-minor-free multigraph $G$ of maximum degree $\Delta$,
  the strong chromatic index of $G$
  satisfies $\chi_{2}'(G) \le 3\Delta$.
\end{thm}
\noindent
The $k=5$ case of Conjecture~\ref{conj:main} includes planar multigraphs, so is a stronger form of~\cite[Conj.~3.2]{LoKa19}.

We have some further justification for Conjecture~\ref{conj:main}. In particular, we show that the construction we give in Section~\ref{sec:diabolical} is asymptotically extremal for a strongly related parameter.
A {\em strong clique} of a multigraph $G$ is a set of edges every pair of which are incident or connected by an edge in $G$;
the {\em strong clique number} $\omega_2'(G)$ of $G$ is the size of a largest such set.
Notice $\omega_2'(G)$ is equal to $\omega(L(G)^2)$, the clique number of the square of the line graph of $G$.
Moreover, $\omega_2'(G) \le \chi_2'(G)$ always.
Our main result is as follows.

\begin{thm}\label{thm:main}
Let $k\ge 4$.
For any $K_k$-minor-free multigraph $G$ of maximum degree $\Delta$, the strong clique number of $G$ satisfies $\omega_2'(G) \le \frac32(k-2)\Delta$.
\end{thm}

\noindent
For each $k\ge 4$, the construction given in Section~\ref{sec:diabolical} matches this bound up to an additive term at most quadratic in $k$. 
Thus Theorem~\ref{thm:main} implies that if there are examples with strong chromatic index asymptotically larger than posited in Conjecture~\ref{conj:main}, then it is due to some global, rather than local, obstruction.

It is important to note here that the bald combination of Hadwiger's Conjecture with Shannon's Theorem yields a bound strictly worse than in Conjecture~\ref{conj:main}; specifically, the bound $\chi_2'(G) \le \frac32(k-1)\Delta$ holds conditionally on the truth of Hadwiger's Conjecture. In contrast to the case of simple graphs (cf.~Theorem~\ref{thm:reduction}), the potential additive discrepancy of $\frac32\Delta$ is intriguing.
In light of the known cases of Hadwiger's Conjecture, the cases $k\in\{5,6\}$ of Conjecture~\ref{conj:main} are especially interesting.
In general the result of Kostochka~\cite{Kos82} implies an unconditional bound of $O(k\sqrt{\log k}\Delta)$. After submission of this manuscript, breakthroughs of Norin, Postle and Song~\cite{NPS20,P20} have yielded an unconditional bound of $O(k (\log k)^{\beta})$ for any $\beta>4$, later improved further to $O(k (\log \log k)^6)$. 

Important to the proof of Theorem~\ref{thm:main} is a reduction (Lemma~\ref{lem:reduction:weighted} below) with which we can restrict attention to the case that the edges of nontrivial multiplicity form an especially well-structured submultigraph.
In fact, the reduction is also highly relevant for the fractional strong chromatic index, as we show in Section~\ref{sec:fractional reduced} (see Conjecture~\ref{conj:fractional reduced} below).

For $k=5$, the construction in Section~\ref{sec:diabolical} is ``very local'', in the sense that \emph{every} two of its edges are incident or joined by an edge, i.e.~it has edge-diameter $2$. In contrast, the following generalisation of~\cite[Prop.~3.3]{LoKa19} shows that any construction (asymptotically) meeting the bound of Theorem~\ref{thm:main} cannot be ``very local'' once $k>5$. 
(Note that $k-\frac12<\frac32(k-2)$ if and only if $k>5$.)

\begin{thm}\label{thm:isclique}
Let $k\ge 5$. For any $K_k$-minor-free multigraph $G$ of maximum degree $\Delta$ such that every two of its edges are incident or joined by an edge, the strong clique number of $G$ satisfies
$\omega_2'(G)  \le (k - \frac12) \Delta.$ 
\end{thm}

\noindent
A \emph{blown-up $5$-cycle} of order $t$ is the graph obtained from the $5$-cycle by substituting each vertex $u$ with a stable set $S_u$ of size $t$ and substituting each edge $uv$ with a complete bipartite graph between $S_u$ and $S_v$.
Note that for odd $t\geq 3$, the blown-up $5$-cycle of order $t$ has no $K_{\frac52t+\frac12}$-minor~(for otherwise, by the pigeonhole principle, at least one branch set of such a minor would need to have size $1$ and thus can neighbour at most $2t< \frac{5t}{2}-\frac{1}{2}$ other branch sets; contradiction), has maximum degree $2t$ and a strong clique of size $5t^2$; this shows that the bound in Theorem~\ref{thm:isclique} (and that of Conjecture~\ref{conj:intermediate}) is exact for infinitely many pairs $(k,\Delta)$.
By a more involved multigraph construction we give in Section~\ref{sec:isclique}, for each $k\ge 5$ and $\Delta \ge k-2$, Theorem~\ref{thm:isclique} is sharp up to an additive term of order at most quadratic in $k$. For $k\le4$, the bound still holds, but then it is not necessarily (asymptotically) sharp, e.g.~Theorem~\ref{thm:K4}.

Returning to just simple graphs, note that a minor adaptation of Theorem~\ref{thm:reduction} (see Theorem~\ref{thm:omnibusreduction} below) yields the following weaker version of Conjecture~\ref{conj:intermediate}.
\begin{thm}\label{thm:simple}
For any $K_k$-minor-free (simple) graph $G$ of maximum degree $\Delta$,
the strong clique number of $G$ satisfies $\omega_2'(G)\le (k-1)(\Delta+1)$.
\end{thm}

\noindent
Although at the end of the paper we discuss how to sharpen this bound slightly, already the blown-up $5$-cycle described just above shows the bound to be sharp up to a $O(\Delta)$ or $O(k)$ additive factor when $\Delta=2t=\frac45(k-\frac12)$ for odd $t$.
For relatively larger $\Delta$, specifically for $\Delta \ge k-2$, a complete bipartite graph $K_{k-2,\Delta}$ having parts of size $k-2$ and $\Delta$, to one vertex in the larger part of which is attached $\Delta-k+2$ pendant edges, is a $K_k$-minor-free graph of maximum degree $\Delta$ with strong clique number $(k-1)(\Delta-1) + 1$.

By another adaptation of Theorem~\ref{thm:reduction}, we also have the following as a corollary of Reed and Seymour's result~\cite{ReSe98} that every $K_k$-minor-free graph has fractional chromatic number at most $2(k-1)$.
The {\em fractional strong chromatic index $\chi_{2,f}'(G)$} of a graph $G$ is $\chi_f(L(G)^2)$.
Note $\omega_2'(G) \le \chi_{2,f}'(G) \le \chi_2'(G)$ always.
\begin{thm}\label{thm:simplefractional}
For any $K_k$-minor-free (simple) graph $G$ of maximum degree $\Delta$,
the fractional strong chromatic index of $G$ satisfies $\chi_{2,f}'(G)\le 2(k-1)(\Delta+1)$.
\end{thm}

\noindent
In fact, improving on the factor $2$ in Theorem~\ref{thm:simplefractional} is equivalent to improving on the factor $2$ in Reed and Seymour's result (see Corollary~\ref{cor:cliquefractionalreduction}).
It is natural here for us to reiterate the importance of Theorem~\ref{thm:reduction} with respect to simple graphs: Conjecture~\ref{conj:intermediate} is an unexpected and new perspective on Hadwiger's Conjecture.

\subsection*{Outline of the paper}
In Section~\ref{sec:reduction} we prove Theorems~\ref{thm:reduction},~\ref{thm:simple} and~\ref{thm:simplefractional}, thus in particular establishing close relationships between (fractional) strong chromatic index and (fractional) chromatic number in minor-closed graph classes. In Section~\ref{sec:diabolical} we construct a multigraph which asymptotically matches Conjecture~\ref{conj:main} (if true) and  Theorem~\ref{thm:main}. Section~\ref{sec:upper} is devoted to the proof of Theorem~\ref{thm:main}. In the process, we develop a reduction tool (see Lemma~\ref{lem:reduction:weighted} and Proposition~\ref{prop:fractional reduced}) that may be of independent interest. In particular, we show in Section~\ref{sec:fractional reduced} that in order to prove the fractional relaxation of Conjecture~\ref{conj:main}, it suffices to restrict our attention to submultigraphs with underlying simple subgraph consisting of a disjoint union of odd cycles and edges.
In Section~\ref{sec:K4} we prove Theorem~\ref{thm:K4} on $K_4$-minor-free multigraphs. In Section~\ref{sec:isclique}, we prove Theorem~\ref{thm:isclique} and we construct a multigraph certifying its asymptotic sharpness. We also prove variants of Theorems~\ref{thm:main} and~\ref{thm:isclique} where the condition of $K_k$-minor-freeness is relaxed to the absence of one particular type of $K_k$-minor. At the end, we revisit strong cliques in simple graphs and we discuss (nonasymptotic) strengthenings of Theorem~\ref{thm:simple}.

\section{Simple reductions}\label{sec:reduction}

We prove an omnibus generalisation of Theorems~\ref{thm:reduction},~\ref{thm:simple} and~\ref{thm:simplefractional} in  Theorem~\ref{thm:omnibusreduction} below.
Note that Theorem~\ref{thm:reduction} follows from parts~\ref{fractionalratiolower} and~\ref{strongratioupper} of Theorem~\ref{thm:omnibusreduction}.
Part~\ref{cliqueratioupper} of Theorem~\ref{thm:omnibusreduction} implies Theorem~\ref{thm:simple}, while part~\ref{cliqueratiolower} certifies that it is asymptotically sharp.
Theorem~\ref{thm:simplefractional} follows from part~\ref{fractionalratioupper} and~\cite{ReSe98}.

For any class ${\cal G}$ of graphs, let us write $\omega({\cal G})$ for the order of the largest clique maximised over the graphs in $\cal G$, and $\omega_2'({\cal G},\Delta)$ (respectively $\chi_{2,f}'({\cal G},\Delta)$)) for the largest strong clique number (respectively fractional strong chromatic index) over the graphs in $\cal G$ of maximum degree $\Delta$.

  \renewcommand{\theenumi}{$(\roman{enumi})$}
  \renewcommand{\labelenumi}{\theenumi}
\begin{thm}\label{thm:omnibusreduction}
Let ${\cal G}$ be a class of graphs.
If ${\cal G}$ is closed under contraction of edges, then
\begin{enumerate}
\item\label{cliqueratioupper} $\omega_2'({\cal G},\Delta) \le \omega({\cal G})(\Delta+1)$,
\item\label{fractionalratioupper} $\chi_{2,f}'({\cal G},\Delta) \le \chi_f({\cal G})(\Delta+1)$, and
\item\label{strongratioupper} $\chi_2'({\cal G},\Delta) \le \chi({\cal G})(\Delta+1)$.
\end{enumerate}
If ${\cal G}$ is closed under attachment of pendant edges, then

\begin{enumerate}
    \setcounter{enumi}{3}
\item\label{cliqueratiolower} $\omega_2'({\cal G},\Delta) \ge \omega({\cal G})\Delta-\binom{\omega({\cal G})}{2}$ if $\Delta > \omega({\cal G})$, and
\item\label{fractionalratiolower} $\chi_{2,f}'({\cal G},\Delta) \ge \chi_f(G)(\Delta-|V(G)|+1)$ for every $G\in {\cal G}$ with $|V(G)|< \Delta$.
\end{enumerate}
\end{thm}

\begin{proof}
We remark that the general argument for~\ref{cliqueratioupper}--\ref{strongratioupper} is essentially the aforementioned cute combination from~\cite{FSGT90}.

For~\ref{cliqueratioupper}, let $G\in {\cal G}$ be a graph of maximum degree $\Delta$ and let $X$ be a largest strong clique in $G$. By Vizing's Theorem, $X$ admits a partition into at most $\Delta+1$ matchings $M_1,\dots,M_{\Delta+1}$ of $G$. For each $i$, consider the graph $G_i$ formed from $G$ by contracting every edge of $M_i$. Since ${\cal G}$ is closed under edge-contraction, $G_i\in {\cal G}$ and so $|M_i| \le \omega(G_i)\le \omega({\cal G})$. We thus have $\omega_2'(G) = |X| \le \sum_i |M_i| \le (\Delta+1)\omega({\cal G})$, as desired.

The arguments for~\ref{fractionalratioupper} and~\ref{strongratioupper} are almost identical, so we only prove~\ref{strongratioupper}. A \emph{partial colouring} of a graph $G$ is a colouring of a subset of its vertices such that adjacent coloured vertices receive distinct colours. Similarly, a \emph{partial strong edge-colouring} of $G$ is a partial colouring of $L(G)^2$. Let $G\in {\cal G}$ be a graph of maximum degree $\Delta$. By Vizing's Theorem, $G$ admits a partition of its edges into at most $\Delta+1$ matchings $M_1,\dots,M_{\Delta+1}$. For each $i$, consider the graph $G_i$ formed from $G$ by contracting every edge of $M_i$. Since ${\cal G}$ is closed under edge-contraction, $G_i\in {\cal G}$ and so there is a proper vertex-colouring of $G_i$ with at most $\chi(G_i) \le \chi({\cal G})$ colours. The partial colouring induced by the vertices corresponding to the contracted edges when transferred directly to the associated edges of $M_i$ in $G$ corresponds to a partial strong edge-colouring of $G$. Combining these $\Delta+1$ partial strong edge-colourings on disjoint colour sets, we obtain a strong edge-colouring of all of $G$ with $(\Delta+1)\chi({\cal G})$ colours, as desired. 

For~\ref{cliqueratiolower}, write $k=\omega({\cal G})$ and consider the graph $K_{k,\Delta}$ formed from $K_k$ by attaching $\Delta-k+1$ new pendant edges to every vertex. By the assumption on ${\cal G}$ and the fact that $K_{k,\Delta}$ has maximum degree $\Delta$, $K_{k,\Delta}\in {\cal G}$ and so it has strong clique number satisfying $\omega_2'(K_{k,\Delta}) \le \omega_2'({\cal G},\Delta)$.
Since $K_{k,\Delta}$ is itself a strong clique, we have that $k(\Delta-k+1)+\binom{k}{2} = \omega_2'(K_{k,\Delta}) \le \omega_2'({\cal G},\Delta)$, as desired.

For~\ref{fractionalratiolower}, consider the graph $G_\Delta$ formed from $G$ by attaching $\Delta-|V(G)|+1$ new pendant edges to every vertex. By the assumption on ${\cal G}$ and the fact that $G_\Delta$ has maximum degree $\Delta$, $G_\Delta\in {\cal G}$ and in particular there is a weighted partition of the set of pendant edges of $G_\Delta$ into induced matchings having total weight at most $\chi_{2,f}'(G_\Delta) \le \chi_{2,f}'({\cal G},\Delta)$.
Note that the vertices of $G$ incident to the edges of each such induced matching form a stable set, and thus giving weight a $1/(\Delta-|V(G)|+1)$ fraction of the weight of each induced matching to its corresponding stable set of $G$ certifies that $\chi_f(G) \le \chi_{2,f}'({\cal G},\Delta)/(\Delta-|V(G)|+1)$, as desired.
\end{proof}

\begin{cor}\label{cor:cliquefractionalreduction}
Let ${\cal G}$ be a class of graphs that is closed under contraction of edges and under attachment of pendant edges.  Then
\[
\limsup_{\Delta\to\infty} \frac{\omega_2'({\cal G},\Delta)}{\Delta} = \omega({\cal G})
\quad\text{ and }\quad
\limsup_{\Delta\to\infty} \frac{\chi_{2,f}'({\cal G},\Delta)}{\Delta} = \chi_f({\cal G}).
\]
\end{cor}

We remark here that almost the same argument as for Theorem~\ref{thm:omnibusreduction}\ref{strongratioupper} easily yields from a result of Kuhn and Osthus~\cite{KuOs03} that Conjecture~\ref{conj:intermediate} holds under the additional assumptions that $G$ has girth at least $9$ and $k$ is sufficiently large.

\section{Conjecturally extremal multigraphs}\label{sec:diabolical}

Here we describe a multigraph construction which asymptotically matches Conjecture~\ref{conj:main} (if true) and  Theorem~\ref{thm:main}.
This is inspired in part by an octahedron-based planar construction given in~\cite{LoKa19}.
It might yet be possible to improve on the strong clique number guaranteed by our construction, but only by an additive amount that is at most of quadratic order in $k$.

For each $k\ge 4$ and each $\Delta \ge k-2$ such that $\Delta+k+1$ is even, we define the multigraph $G_{k,\Delta}$ as follows.
The vertex set $V_{k,\Delta}$ is the disjoint union $V_{k,\Delta}=A\cup B\cup C\cup\{a,a',b,b',c,c'\}$, where 
$A=\{a_1,\dots,a_{k-4}\}$, $B=\{b_1,\dots,b_{k-4}\}$, $C=\{c_1,\dots,c_{k-4}\}$. 
Note that the sets $A, B$ and $C$ are empty if $k=4$. For the edge set $E_{k,\Delta}$, we have the following. For each $1\le i \le k-4$, $\{a_i,b_i,c_i\}$ induces a Shannon triangle of edge multiplicity $(\Delta-(k-3))/2$, and $a_i$ is connected by simple edges to $b$ and $c$, $b_i$ to $a$ and $c$, and $c_i$ to $a$ and $b$. In addition, $ab,bc$ and $ac$ are simple edges. Moreover, each of the vertex subsets $\{a,b,c\}$, $A$, $B$, $C$ induces a simple clique, and the edges $aa'$, $bb'$, $cc'$ are each of multiplicity $\Delta-2(k-4)-2=\Delta-2(k-3)$.
See Figure~\ref{fig:diabolical}.

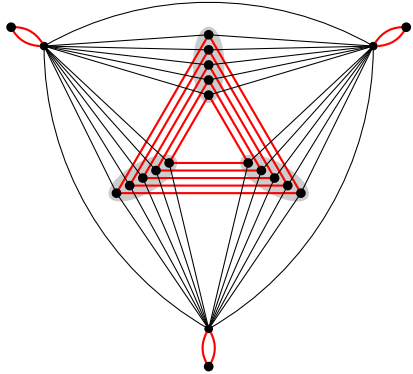
\begin{figure}[!ht]
  \begin{center}

 \begin{tikzpicture}[scale=0.1]
       \tikzstyle{v}=[black,circle,draw,fill=black,scale=0.3];
       \tikzstyle{edge}=[];
       \tikzstyle{clique}=[thick, red];
       \foreach\angle in { 90, 210, 330 } {
         \draw[gray!40,fill, rotate=\angle] (10, 0) ellipse (5 and 2);
       }
       \foreach\r in {6,8,...,14}{
         \draw[clique]
         (-30:\r)node[v](A\r){}--(90:\r)node[v](B\r){}--(210:\r)node[v](C\r){}--cycle;
       }
       \foreach\name/\angle in {C/30, B/-90, A/-210} {
         \draw[edge] (\angle:25) node[v](\name){} arc ({30+\angle}:{90+\angle}:{50*cos(30)});
         \draw[clique, bend left] (\name) to (\angle:30) node[v]{} to (\name);
       }
       \foreach\X/\Y in {A/B, B/C, C/A} {
         \foreach\r in {6,8,...,14} {
           \draw[edge] (\X) -- (\Y\r);
           \draw[edge] (\Y) -- (\X\r);
         }
       }
     \end{tikzpicture}
  \end{center}
  \caption{A schematic of the graph $G_{k,\Delta}$. The grey regions indicate cliques. The edges coloured red are of nontrivial multiplicity, indeed of roughly $\Delta/2$.}\label{fig:diabolical}
\end{figure}

It should be noted that the edges of $G_{k,\Delta}$ itself do not form a strong clique (unless $k$ equals $4$ or $5$), but a suitably large subset of its edges does, as we will now make explicit. Let $Q_{k,\Delta}$ denote the set of all edges of $G_{k,\Delta}$ except those in the three cliques induced by $A$, $B$, and $C$.
Observe that $Q_{k,\Delta}$ is a strong clique in $G_{k,\Delta}$ and 
\begin{align*}
|Q_{k,\Delta}|
&=\frac32(k-4)(\Delta-(k-3))+3(\Delta-2(k-3))+3(2(k-4)+1)\\
&=\frac32(k-2)\Delta-\frac32(k^2-7k+14).
\end{align*}

\begin{prop}\label{prop:diabolical}
For each $k\ge 4$ and each $\Delta \ge k-2$ such that $\Delta+k+1$ is even, the multigraph $G_{k,\Delta}$ is $K_k$-minor-free.
\end{prop}

\begin{proof}
  Assume for a contradiction that Proposition~\ref{prop:diabolical}
  does not hold and let~$k\ge 4$ be the minimal integer such
  that~$G_{k,\Delta}$ contains $K_k$ as a minor for some admissible $\Delta$.
  As the multigraph~$G_{4,\Delta}$ is a triangle with pendant multiedges,
  and therefore it has no~$K_4$ as a minor,
  we know that~$k\ge 5$.
  
  Let $S_1,\dots,S_k\subseteq V_{k,\Delta}$ be a family of disjoint connected
  sets of vertices that, when contracted, induce a~$K_k$.
  As~$G_{k,\Delta}$ is connected, we may assume that~$\bigcup_{i=1}^kS_i=V_{k,\Delta}$.
  
  For $i\in\{1,\dots,k-4\}$,
  let us write~$T_i$ for the Shannon triangle~$\{a_i,b_i,c_i\}$.
  \begin{claim}\label{claim:no triangle}
    $T_i \nsubseteq S_j$ for every $i\in\{1,\dots,k-4\}$ and $j\in\{1,\dots,k\}$.
  \end{claim}
  \begin{claimproof}
    If otherwise $T_i \subseteq S_j$ for some $i, j$,
    then the family $\sst{S_\ell}{\ell\neq j, 1\le j \le k}$ forms
    a~$K_{k-1}$ minor of~$G_{k,\Delta}\setminus T_i$. Since the underlying simple graphs of $G_{k,\Delta}\setminus T_i$ and $G_{k-1,\Delta-1}$ are isomorphic,
    this contradicts the minimality of~$k$.
  \end{claimproof}
  \begin{claim}\label{claim:no edge}
    $|T_i \cap S_j|\le 1$ for every $i\in\{1,\dots,k-4\}$ and $j\in\{1,\dots,k\}$.
  \end{claim}
  \begin{claimproof}
    We know from Claim~\ref{claim:no triangle} that $|T_i \cap S_j|\le 2$,
    so assume that~$T_i \cap S_j$ contains two elements, say $a_i$ and $b_i$.
    Assume that~$c_i\in S_{j'}$ where $j\neq j'$.

    First note that~$S_{j'}\neq\{c_i\}$ because $c_i$ has only $k-3$
    neighbours in~$V_{k,\Delta}\setminus S_{j}$ and
    there must be an edge between $S_{j'}$ and every other $S_{j''}$.
    
    Now note that every neighbour~$x$ of~$c_i$ with $x \notin \{a_i,b_i\}$ satisfies
    $ N(x)\supseteq N(c_i)\setminus S_j$.
    It follows first that the set~$S_{j'}\setminus\{c_i\}$ is connected
    and second that~$S_{j'}\setminus\{c_i\}$ is adjacent to the set~$S_\ell$
    whenever $\ell\in\{1,\dots,k\}\setminus\{j,j'\}$. (Throughout the paper, when we say that two vertex sets $X$, $Y$ are \emph{adjacent}, we mean that there exists an edge $xy$ with $x\in X$ and $y\in Y$.)

    As a consequence, the family
    $\sst{S_\ell}{\ell\neq j,\ell\neq j', 1\le j \le k}\cup\{S_{j'}\setminus\{c_i\}\}$ forms
    a~$K_{k-1}$-minor of~$G_{k,\Delta}\setminus T_i=G_{k-1,\Delta}$,
    which contradicts the minimality of~$k$.
  \end{claimproof}
  \begin{claim}\label{claim:ABC vs edges}
    For every $j\in\{1,\dots,k\}$, the set $S_j$ intersects at most one of~$A\cup B\cup C$
    and~$\{a,b,c\}$.
  \end{claim}
  \begin{claimproof}
    The proof of this claim is similar to the one of Claim~\ref{claim:no edge}.
    Assume otherwise, then by symmetry and because~$S_j$ is connected,
    we may assume that~$S_j$ contains $a$ and~$b_i$ for some~$i\in\{1,\dots,k-4\}$.
    Assume that~$a_i\in S_{j_1}$ and~$c_i\in S_{j_2}$.
    By Claim~\ref{claim:no edge}, $j$, $j_1$ and~$j_2$ are distinct.

    We aim to show that the family $\sst{S_\ell\setminus T_i}{\ell\neq j_1}$
    forms a~$K_{k-1}$-minor of~$G_{k,\Delta}\setminus T_i$,
    contradicting the minimality of~$k$.
    
    First note that~$S_{j_2}\setminus T_i = S_{j_2}\setminus\{c_i\}$
    is nonempty because $c_i$ has only~$k-3$ neighbours
    in~$V_{k,\Delta}\setminus S_{j}$.
    Moreover, every neighbour~$x$ of~$c_i$ in~$V_{k,\Delta}\setminus T_i$
    satisfies~$N(x)\supseteq N(c_i)\setminus T_i$, so $S_{j_2}\setminus\{c_i\}$
    is connected and is adjacent to every~$S_{\ell}\setminus T_i$
    for~$\ell\notin\{j_1,j_2\}$.

    Similarly, every neighbour of~$b_i$ in~$V_{k,\Delta}\setminus (T_i\cup \{a\})$
    is a neighbour of~$a$, so~$S_j\setminus T_i$ is connected and adjacent
    to every $S_\ell$ with~$\ell\notin\{j,j_1\}$.
    This suffices to show that $\sst{S_\ell\setminus T_i}{\ell\neq j_1}$
    forms a~$K_{k-1}$-minor and prove the claim.
  \end{claimproof}
  \begin{claim}\label{claim:separation}
    For every~$j\in\{1,\dots,k\}$ the set~$S_j$
    is included in one of the sets~$A$, $B$, $C$ or~$\{a,a',b,b',c,c'\}$.
  \end{claim}
  \begin{claimproof}
    We know from Claim~\ref{claim:ABC vs edges}
    that $S_j$ is a subset of~$A\cup B\cup C$ or~$\{a,a',b,b',c,c'\}$.
    In the case where $S_j\subseteq A\cup B\cup C$,
    the claim follows from Claim~\ref{claim:no edge}
    and the fact every edge between two sets of $A$, $B$ and~$C$ is part of
    a triangle~$T_i$.
  \end{claimproof}
  \begin{claim}\label{claim:S0}
    There exists~$j_0$ such that~$S_{j_0}=\{a,a',b,b',c,c'\}$.
  \end{claim}
  \begin{claimproof}
    The set $S_{j_0}$ that contains $a$ is adjacent to the set $S_j$ that contains
    $a_1$.
    By Claim~\ref{claim:separation}, this is possible only if~$S_{j_0}$
    also contains one of $b$ and~$c$.
    Applying the same argument to~$b$ and~$c$ gives the claim.
  \end{claimproof}
  Let~$k_a$, $k_b$ and~$k_c$ be the number of sets~$S_i$ included in~$A$, $B$ and~$C$, respectively.
  \begin{claim}\label{claim:bridges inequality}
    $k_ak_b\le k-4$, $k_bk_c\le k-4$ and~$k_ak_c\le k-4$.
  \end{claim}
  \begin{claimproof}
    There is at least one edge between each $S_i\subseteq A$ and~$S_j\subseteq B$,
    so there are at least~$k_ak_b$ edges between~$A$ and~$B$.
    This proves that~$k_ak_b\le k-4$. The two other cases are symmetric.
  \end{claimproof}
  We are now ready to finish the proof.
  Note that we have~$k=k_a+k_b+k_c+1$.
  As $k\ge 5$, we can assume that for instance~$k_a\ge 2$.
  By Claim~\ref{claim:bridges inequality},
  \begin{equation}\label{eq:bound on k}
    k = k_a + k_b + k_c + 1 \le k_a + 2\frac{k - 4}{k_a} + 1.
  \end{equation}
  The right side of~\eqref{eq:bound on k} is a convex function of~$k_a$,
  so it is maximum on the boundary of the domain.
  As $2\le k_a \le k-4$, it suffices for the final contradiction to show that~\eqref{eq:bound on k}
  does not hold when~$k_a\in\{2, k-4\}$.

  Indeed, if~$k_a$ is equal to~$2$ or $k-4$, then
  \[
    k_a + 2\frac{k - 4}{k_a} + 1 = k-1. \qedhere
  \]
\end{proof}

\section{A bound on strong clique number}\label{sec:upper}

This section is devoted to proving Theorem~\ref{thm:main}.
A critical tool in the proof is the following reduction.
This applies more generally whenever we wish to upper bound the strong clique number of multigraphs. It allows us to restrict our attention to multigraphs where the edges of nontrivial multiplicity induce (in the underlying simple graph) a vertex-disjoint union of odd cycles and edges.

Given a graph~$G=(V,E)$ and a weight~$w:E\to\mathbb{R}$ on the edges of~$G$,
we define $d_w(v)=\sum_{e\ni v}w(e)$ and~$\Delta(w)=\max_{v\in V}d_w(v)$.
The \emph{support} of~$w$ is the set $\supp w$
of edges~$e\in E$ for which~$w(e)\neq 0$.
If~$H$ is a subgraph of $G$, we write $w(H)$ for the sum $\sum_{e\in E(G)}w(e)$.
\begin{lemma}\label{lem:reduction:weighted}
  Let $G = (V,E)$ be a simple graph
  and $w:E \to \mathbb{R}_{0}^{+}$ a nonnegative weight on the edges of~$G$.
  There exist nonnegative weights~$w_1,\dots,w_p$ on~$E$
  such that
  \renewcommand{\theenumi}{$(\roman{enumi})$}
  \renewcommand{\labelenumi}{\theenumi}
  \begin{enumerate}
  \item\label{itm:red1}
    $\sum_{i=1}^pw_i=w$,
  \item\label{itm:red2}
    $\sum_{i=1}^p\Delta(w_i) = \Delta(w)$,
  \end{enumerate}
  and for every~$i\in\{1,\dots,p\}$, the following hold.
  \begin{enumerate}
    \setcounter{enumi}{2}
  \item\label{itm:red3}
    The support of~$w_i$ induces a vertex-disjoint union of odd cycles and edges in
    $G$.
  \item\label{itm:red4}
    For every edge $e$ in the support of~$w_i$,
    $w_i(e)=\frac1{2}{\Delta(w_i)}$ if~$e$ is part of a cycle (of odd length) of~$\supp w_i$
    and~$w_i(e)=\Delta(w_i)$ otherwise.
  \end{enumerate}
\end{lemma}

\begin{proof}
  The proof proceeds by induction on the number of edges in the support of~$w$.
  If $w$ satisfies~\ref{itm:red3} and~\ref{itm:red4}, then the trivial decomposition $w=w_1$ fits the lemma. So we may assume that $w$ does not satisfy one of~\ref{itm:red3} and~\ref{itm:red4}.
  
  It will be sufficient to show that 
  $w$ can be written as~
  \begin{equation}~\label{dec:w}
  w=w_1+w_2  
  \end{equation} 
  for some nonnegative weights $w_1$ and $w_2$ 
  such that~$\Delta(w_1)+\Delta(w_2)=\Delta(w)$ and for~$i\in\{1,2\}$,
  either~$|\supp w_i| < |\supp w|$
  or~$w_i$ satisfies~\ref{itm:red3} and~\ref{itm:red4}.

Indeed, if we succeed in showing this, then for each $i\in \{1,2\}$, it follows by induction that $w_i$ can be decomposed as $w_i=\sum_{j=1}^{p_i} w_{i_j}$ for some nonnegative weights $w_{i_1}, \ldots, w_{i_{p_i}}$ satisfying \ref{itm:red3},~\ref{itm:red4} and $\sum_{j=1}^{p_i}\Delta(w_{i_j}) = \Delta(w_i)$. Then $w_{1_1}, \ldots, w_{1_{p_1}}$ and $w_{2_1},\ldots, w_{2_{p_2}}$ together form the desired weights, finishing the proof.

Next, we argue that to obtain a decomposition as in (\ref{dec:w}), it suffices to construct a nonzero weight $t:E\to \mathbb{R}$
  (with positive and negative values) with $\supp t \subseteq \supp w$,
  that satisfies the following.
  Letting~$X_t$ be the set of vertices~$v$ with~$d_t(v)\neq 0$,
  then
  \renewcommand{\theenumi}{$(\alph{enumi})$}
  \renewcommand{\labelenumi}{\theenumi}
  \begin{enumerate}
  \item\label{it:t:degree1}
    every vertex of~$X_t$ is incident to exactly one edge of~$\supp w$; and
  \item\label{it:t:stable}
    $t(uv)=0$ whenever $uv$ is an edge with~$u,v\in X_t$.
  \end{enumerate}
  Note that $X_t$ could be empty. Before identifying the cases in which we can construct $t$, we first show that we can indeed use such a $t$ to obtain the desired decomposition of $w$ as in~(\ref{dec:w}). 
  Let~$m_1$ be the largest real number such that~$w + m_1\cdot t$ has no negative
  value, and, similarly, let~$m_2$ be the largest real number
  such that $w - m_2\cdot t$ has no negative value.
  
  To ensure that~$m_1$ and~$m_2$ are well-defined, we need to check that~$t$
  has at least one positive and one negative value.
  To see this, note that $t$ is nonzero so there is an $e\in E$ with~$t(e)\neq 0$.
  By~\ref{it:t:stable}, $e$ is incident to a vertex~$v \notin X_t$
  with $d_t(v)=0$, which implies that~$v$ is incident to at least one edge~$e'$
  such that the sign of $t(e')$ is the opposite of the sign of~$t(e)$.
  
  Now define
  \[
    w_1 = \frac{m_2}{m_1+m_2}(w + m_1t)
    \text{\quad and \quad}
    w_2 = \frac{m_1}{m_1+m_2}(w - m_2t).
  \]
  It is clear from this definition that~$w=w_1+w_2$.
  For $i\in\{1,2\}$,
  the definition of~$m_i$ also ensures that $w_i$ has no negative value
  and satisfies $\supp w_i \subsetneq \supp w$
  (recalling that $\supp t\subseteq \supp w$).
  So it remains to show that~$\Delta(w_1)+\Delta(w_2)=\Delta(w)$.

  For~$\lambda\in\{m_1,0,-m_2\}$, define $w_\lambda=w+\lambda t$.
  Note that~$\supp w_\lambda \subseteq \supp w$.
  By~\ref{it:t:degree1},
  every vertex~$u\in X_t$ is incident to exactly one edge $uv$ of~$\supp w$.
  In particular
  $d_{w_\lambda}(u) = w_\lambda(uv) \le d_{w_\lambda}(v)$.
  By~\ref{it:t:stable}, we also know that~$v\notin X_t$.
  As a consequence, the maximum degree of~$w_\lambda$ can be
  expressed as
  \[
    \Delta(w_\lambda)=\max_{v\in V\setminus X_t}d_{w_\lambda}(v)
    =\max_{v\in V\setminus X_t}(d_w(v) + \lambda\cdot d_t(v))
    =\max_{v\in V\setminus X_t}d_{w}(v),
  \]
  where the last step uses the definition of~$X_t$.
  As this last expression does not depend on~$\lambda$,
  it follows that~$\Delta(w+m_1t)=\Delta(w-m_2t)=\Delta(w)$.
  This is enough to conclude that
  $\Delta(w_1)+\Delta(w_2)=
  m_2/(m_1+m_2)\cdot\Delta(w + m_1t) + m_1/(m_1+m_2)\cdot\Delta(w - m_2t)
  =\Delta(w)$.
This concludes the proof that it suffices to construct the weight $t$.

  We now describe four cases in which we can indeed construct $t$.
  See Figure~\ref{fig:transformations} for an artist's depiction
  of these constructions.
  
  For case one, assume $\supp w$ contains an even cycle $C=v_0\dots v_{2j-1}$.
  Define $t(v_{2i}v_{2i+1})=1$ and $t(v_{2i+1}v_{2i+2})=-1$
  (with indices modulo $2j$) for every~$i\in\{0,\dots,j-2\}$,
  and $t(e)=0$ when~$e\in E \backslash C$.
  It is clear that $\supp t=C\subseteq\supp w$.
  Moreover, $X_t=\emptyset$, so~\ref{it:t:degree1} and~\ref{it:t:stable}
  are trivially satisfied.
  
  For case two,
  suppose there is a path $P$ in $\supp w$, possibly of length zero,
  that connects two distinct odd cycles $C, C'$ in $X'$.
  Assuming~$C=v_0\dots v_{2k}$, where $v_0$ is the common vertex of~$P$ and~$C$,
  set $e_0=v_{k}v_{k+1}$.
  We define the weight $t$
  on an edge $e\in P\cup C\cup C'$ depending on the parity of
  the distance $d$ from $e_0$ to $e$ through the edges of $P\cup C\cup C'$.
  Define $t(e)=(-1)^{d+1}$ if  $e\in C\cup C'$,  $t(e)=2\cdot (-1)^d$ if $e\in P$. Furthermore, define $t(e)=0$ if $e \notin P \cup C \cup C'$.
  Note that $\supp t = P\cup C\cup C'$ is indeed a subset of
  the support of~$w$.
  Again, $X_t$ is empty, so~\ref{it:t:degree1} and~\ref{it:t:stable}
  are trivial.
  
  For case three,
  suppose there is a path $P=v_0\dots v_{j}$ of length at least~$2$ in
  the support of~$w$
  such that~$v_0$ and~$v_j$ are incident to exactly one edge of~$\supp w$ each
  (namely $v_0v_1$ and~$v_{j-1}v_j$).
  Define $t(v_{i}v_{i+1})=(-1)^i$ whenever $i\in\{0,\dots,j-1\}$,
  and~$t(e)=0$ for every~$e\notin P$.
  In this case, $X_t=\{v_0, v_j\}$,
  so~\ref{it:t:degree1} is part of our assumptions, and~\ref{it:t:stable}
  is ensured by the fact that~$j\ge 2$.

  For case four, suppose that there is a path $P$ in $\supp w$ of nonzero length,
  one endpoint of which is connected to an odd cycle $C$ in $\supp w$,
  the other endpoint~$u$ of which is incident to no edge of $\supp w$ outside of $P$. Let $e_0$ be the unique edge of $P$ that is incident to $u$.
  Then define~$t(e)=-2$ if~$e\in P$ is at even distance from~$e_0$,
  and~$t(e)=2$ otherwise.
  Define~$t(e)=1$ if~$e\in C$ is at even distance from~$e_0$,
  and~$t(e)=-1$ otherwise. Furthermore, set $t(e)=0$ if $e\notin P \cup C$.
  In this case, $X_t=\{u\}$, so~\ref{it:t:degree1} is part of our assumptions
  and~\ref{it:t:stable} is trivially satisfied.

From now on we can assume that $\supp w$ contains none of the four structures described in the four cases above (otherwise we are done). From this, we will directly derive a decomposition as in~(\ref{dec:w}).
  The odd cycles in $\supp w$ are edge-disjoint,
  for otherwise case one holds.
  A connected component of the graph $(V, \supp w)$ cannot contain
  two odd cycles as otherwise case two holds.
  Moreover, a connected component of $(V, \supp w)$
  that contains a vertex incident to only one edge of~$\supp w$ is
  composed of a single edge, for otherwise
  case three or four holds.
  This verifies that $w$ satisfies Property~\ref{itm:red3}.

  Thus we know that $w$ does not satisfy~\ref{itm:red4} (otherwise we are done, as argued at the start of the proof).
  Now let~$t$ be the weight with $\supp t=\supp w$
  defined on its support by~$t(e)=1$ if~$e$ is part of an odd cycle of $\supp w$ and
  $t(e)=2$ otherwise (i.e. if~$e$ induces a connected component of $(V, \supp w)$).
  Let $m$ be the largest real number such that $m\cdot t(e)\le w(e)$
  for every $e\in E$.
  Then define
  \[
    w_1=m\cdot t \text{\quad and \quad} w_2 = w-w_1=w-m\cdot t.
  \]
  It is clear that~$w=w_1+w_2$.
  Since $d_t(u)=2$ for every~$u$ incident to an edge of $\supp w$,
  it holds that $\Delta(w_1)+\Delta(w_2)=2m + (\Delta(w) - 2m)=\Delta(w)$.
  Now, $w_1$ satisfies~\ref{itm:red3} and~\ref{itm:red4}
  and $w_2$ satisfies $|\supp w_2| < |\supp w|$
  as a consequence of the definition of~$m$.
  Thus $w_1$ and $w_2$ form a decomposition of $w$ as described in~(\ref{dec:w}), as desired.
\end{proof}

  \begin{figure}[t]
  \begin{center}
    \tikzstyle{v}=[black,circle,draw,fill=black,scale=0.5]
    \tikzstyle{clique}=[very thick]
    \begin{tabular}{p{2.5cm} p{2.5cm} p{2.5cm} p{2.5cm}}
      \begin{tikzpicture}[baseline=20]
        \draw[clique]
        (0,0)node[v]{} -- node[midway,left]{$-1$}
        (0,1)node[v]{} -- node[midway,above]{$1$}
        (1,1)node[v]{} -- node[midway,right]{$-1$}
        (1,0)node[v]{} -- node[midway,below]{$1$}
        (0,0)node[v]{};
      \end{tikzpicture} &
      \begin{tikzpicture}[baseline = -20]
        \draw[clique]
        (0,1) node[v](A){} -- node[midway, below right]{$1$}
        (.5,1.87) node[v]{} -- node[midway, above]{$-1$}
        (-.5,1.87) node[v]{} -- node[midway, below left]{$1$}
        (A) -- node[midway,left]{$-2$}
        (0,0) node[v]{} -- node[midway,left]{$2$}
        (0,-1) node[v]{} -- node[midway,above left]{$-1$}
        ($(0,-2)+(162:1)$) node[v]{} -- node[midway,left]{$1$}
        ($(0,-2)+(234:1)$) node[v]{} -- node[midway,below]{$-1$}
        ($(0,-2)+(306:1)$) node[v]{} -- node[midway,right]{$1$}
        ($(0,-2)+(18:1)$) node[v]{} -- node[midway,above right]{$-1$}
        (0,-1);
      \end{tikzpicture} & \quad
      \begin{tikzpicture}[baseline = 50]
        \draw[clique]
        (0,0)node[v, blue]{}
        -- node[midway,left]{$-1$}(0,1)node[v]{}
        -- node[midway,left]{$1$}(0,2)node[v]{}
        -- node[midway,left]{$-1$} (0,3)node[v]{}
        -- node[midway,left]{$1$} (0,4) node[v, blue]{};
      \end{tikzpicture}& 
      \begin{tikzpicture}[baseline = 0]
        \draw[clique]
        (0,-1) node[v](A){} -- node[midway, above right]{$1$}
        (.5,-1.87) node[v]{} -- node[midway, below]{$-1$}
        (-.5,-1.87) node[v]{} -- node[midway, above left]{$1$}
        (A) -- node[midway,left]{$-2$}
        (0,0) node[v]{} -- node[midway,left]{$2$}
        (0,1) node[v]{} -- node[midway,left]{$-2$}
        (0,2) node[v, blue]{};
      \end{tikzpicture}\\
    \end{tabular}
  \end{center}
  \caption{Examples of each of the four types of construction of~$t$
    in the proof of Lemma~\ref{lem:reduction:weighted}.
    Vertices in blue are adjacent to exactly one edge of the support of~$w$.}
  \label{fig:transformations}
  \end{figure}
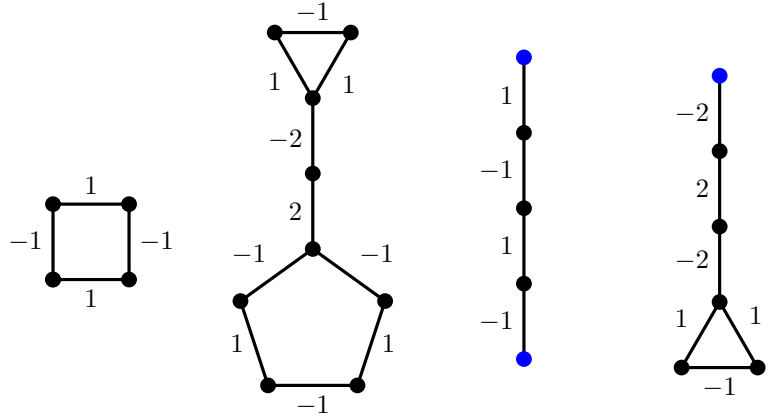

\begin{proof}[Proof of Theorem~\ref{thm:main}]
  Let $G$ be a $K_k$-minor-free multigraph of maximum degree $\Delta$.
  Let $X$ be a strong clique in $G$.

  Let~$H$ be the simple graph underlying~$G$.
  Consider the weight $w$ on the edges of~$H$ where $w(e)$ is the multiplicity
  of~$e$ in~$X$.
  It follows from this definition that $\supp(w)$ is a strong clique in $H$ with weight $w(H)=|X|$.
  Let $w_1,\dots,w_\ell$ be as guaranteed by Lemma~\ref{lem:reduction:weighted}
  upon the input of $H$ and $w$.
  Let~$j$ be the index that maximises the ratio $\frac{w_j(H)}{\Delta(w_j)}$,
  and set $t=\frac{2}{\Delta(w_j)}w_j$.
  It holds that
  \begin{align*}
    w(H)&=\sum_{i=1}^\ell w_i(H)
          =  \sum_{i=1}^\ell\frac{w_i(H)}{\Delta(w_i)}\Delta(w_i)\\
        &\le \frac{w_{j}(H)}{\Delta(w_j)}\sum_{i=1}^\ell\Delta(w_i)
          = \frac{w_j(H)}{\Delta(w_j)}\Delta(w)=\frac{\Delta(w)}{2}t(H),         
  \end{align*}
  so, using that $w(H)=|X|$ and $\Delta(w)\le\Delta$,
  \begin{equation}\label{eqn:X le t}
    |X| \le \frac{\Delta}{2}t(H).         
  \end{equation}  
  Let $Q$ be the support of~$t$. As a subset of the strong clique $\supp w$,
  the set $Q$ also induces a strong clique of $H$.
  The weight $t$ satisfies~\ref{itm:red3} and~\ref{itm:red4}
  because these properties are preserved by multiplication with a scalar.
  Since $\Delta(t)=2$,
  it follows that $Q$ is a vertex-disjoint
  union of subgraphs $\{{\cal D}_1,\dots,{\cal D}_{\ell}\}$
  that are each either a cycle
  of odd length or a single edge and that the value $t({\cal D}_i)$
  is the number of edges in ${\cal D}_i$ if ${\cal D}_i$ is an odd cycle,
  and $t({\cal D}_i)=2$ if ${\cal D}_i$ is a single edge.

  By~\eqref{eqn:X le t} it suffices to show that $t(H) \le 3(k-2)$,
  so let us assume otherwise for a contradiction. Recall that $L({\cal D}_i)$ denotes the line graph of ${\cal D}_i$ and $\alpha(L({\cal D}_i))$ denotes the size of a maximum independent set within it.
  For each $i$, let us define
  \[
    \zeta({\cal D}_i) = 3\alpha(L({\cal D}_i)) - t({\cal D}_i)
  \]
  and write $\zeta(Q)=\sum_{i=1}^{\ell} \zeta({\cal D}_i)$
  so that
  \[
    t(H) =
    \sum_{i=1}^{\ell}t({\cal D}_i) =
    \left( 3\sum_{i=1}^{\ell} \alpha(L({\cal D}_i)) \right)-\zeta(Q).
  \]
  
  If $K_2$ denotes a single edge and $C_{2j+1}$ denotes a cycle of length $2j+1$ for some $j\ge 1$, then
  \begin{align}\label{eqn:valuesshortcycles}
  \zeta(K_2)=1 \quad\text{and}\quad \zeta(C_{2j+1})=j-1.
\end{align}
  In particular, $\zeta({\cal D}_i)$ is nonnegative for each $i$ and so
  \[
    3\sum_{i=1}^{\ell} \alpha(L({\cal D}_i)) \ge t(H) >  3(k-2),
  \]
implying $\sum_{i=1}^{\ell} \alpha(L({\cal D}_i)) \ge k-1$.
A union of stable sets of the graphs $L({\cal D}_i)$ corresponds to a set of pairwise nonincident edges in the strong clique $Q$ of $H$. Thus these edges form a strong clique of $H$ as well, and so correspond also to a clique minor of $H$.
We may therefore assume that $\sum_{i=1}^{\ell} \alpha(L({\cal D}_i)) = k-1$.
This yields
\[
  t(H)= 3(k-1)-\zeta(Q) = 3(k-2)+(3-\zeta(Q)).
\]
Thus it remains only to derive a contradiction under the assumption that 
\begin{align}\label{eqn:thecase}
  \zeta(Q) < 3.
\end{align}

From~\eqref{eqn:valuesshortcycles} it follows that~\eqref{eqn:thecase} only occurs in one of the following seven cases for the composition of ${\cal D}_1,\dots,{\cal D}_{\ell}$.

\renewcommand{\theenumi}{$(\alph{enumi})$}
\renewcommand{\labelenumi}{\theenumi}
\begin{enumerate}
\item\label{allC3s} $k-1$ triangles. ($\zeta(Q)=0$.)
\item\label{oneC2} One $K_2$ and $k-2$ triangles. ($\zeta(Q)=1$.)
\item\label{oneC5} One $C_5$ and $k-3$ triangles. ($\zeta(Q)=1$.)
\item\label{twoC2} Two $K_2$ and $k-3$ triangles. ($\zeta(Q)=2$.)
\item\label{oneC2oneC5} One $K_2$, one $C_5$, and $k-4$ triangles. ($\zeta(Q)=2$.)
\item\label{twoC5} Two $C_5$ and $k-5$ triangles. ($\zeta(Q)=2$.)
\item\label{oneC7} One $C_7$ and $k-4$ triangles. ($\zeta(Q)=2$.)
\end{enumerate}
To complete the proof we will show that in each of these cases $H$ contains $K_k$ as a minor.

Since $k\geq 4$ by assumption of the theorem and because the union of ${\cal D}_1,\ldots, {\cal D}_{\ell}$ is a strong clique, in each of the cases~\ref{allC3s}--\ref{oneC7}, we either have at least one triangle or $K_k$ is a minor (or both). So from now on we may assume that there is at least one triangle. 
To efficiently use the triangles for our construction of $K_k$ as a minor, it is important to make the following simple structural observation.

\resetClaimCounter
\begin{claim}\label{clm:C3s}
For any two vertex-disjoint triangles $\tau$ and $\tau'$ in the same strong clique, one of the following two statements holds.
\renewcommand{\theenumi}{(T$\arabic{enumi}$)}
\renewcommand{\labelenumi}{\theenumi}
\begin{enumerate}
\item\label{dominate}
For one of the triangles, say $\tau$, each of its vertices is adjacent to a vertex in the other triangle $\tau'$. In this case, we say $\tau$ {\em dominates} $\tau'$.
\item\label{complete}
Some two vertices in $\tau$ and two vertices in $\tau'$ together induce a clique.
\end{enumerate}
\end{claim}

\begin{claimproof}
If statement~\ref{dominate} does not hold, then there must be a vertex $u$ in one triangle and a vertex $v$ in the other that are not adjacent. Since the triangles are in the same strong clique, by considering in pairs the four edges of the triangles incident to the $u$ and $v$, the remaining four vertices induce a clique.
\end{claimproof}

Let ${\mathscr T}=\{\tau_1,\dots, \tau_t\}$ denote the set of triangles from $\{{\cal D}_1,\dots,{\cal D}_{\ell}\}$.
Based on the statements from Claim~\ref{clm:C3s}, we form an auxiliary mixed graph $G_{\mathscr T} = ({\mathscr T},E_{\mathscr T},A_{\mathscr T})$, so with both undirected and directed edges, from ${\mathscr T}$ as follows.
If $\tau_i$ dominates $\tau_j$ as in~\ref{dominate} then direct an edge from $\tau_i$ to $\tau_j$, i.e.~$(\tau_i,\tau_j)\in A_{\mathscr T}$.
If $\tau_i$ and $\tau_j$ satisfy statement~\ref{complete} of Claim~\ref{clm:C3s}, then include an edge between $\tau_i$ and $\tau_j$, i.e.~$\tau_i\tau_j \in E_{\mathscr T}$.
Note that for each pair $i,j$ any combination of $(\tau_i,\tau_j)\in A_{\mathscr T}$, $(\tau_j,\tau_i)\in A_{\mathscr T}$, and $\tau_i\tau_j\in E_{\mathscr T}$ apart from the empty combination (by Claim~\ref{clm:C3s}) can occur.

The rest of the proof relies on a structural partition ${\mathscr P}$ of $G_{\mathscr T}$.
Iteratively, for each $i\ge 1$, let $S_i \subseteq {\mathscr T}\setminus \cup_{j=1}^{i-1} S_j$ be an arbitrary maximal vertex subset such that, in the directed subgraph of $({\mathscr T},A_{\mathscr T})$ induced by $S_i$, from some root vertex $\rho_i$ there is a directed path to any other vertex of $S_i$. (Stop the iteration as soon as $\cup_{j\geq 1} S_j$ covers all of ${\mathscr T}$.)
Let $\Arb_i$ denote the corresponding spanning arborescence of $G_{\mathscr T}[S_i]$ rooted at $\rho_i$.
Note that each $S_i$ and $\Arb_i$ has nonzero size, since it must at least contain some $\rho_i$.
So we write ${\mathscr P}=(S_1,\dots,S_p)$ for some $1\le p\le t$.
It is of course possible for the $S_i$ to consist of $t$ singleton sets.

\begin{claim}\label{clm:roots}
$\cup_{i=1}^p\{\rho_i\}$ induces a stable set in $({\mathscr T},A_{\mathscr T})$ and a clique in $({\mathscr T},E_{\mathscr T})$.
\end{claim}
\begin{claimproof}
If it does not induce a stable set in $({\mathscr T},A_{\mathscr T})$, then we would have had a larger choice of subset $S_i$ for some $i$, contradicting maximality. That it induces a clique in $({\mathscr T},A_{\mathscr T})$ then follows from Claim~\ref{clm:C3s}.
\end{claimproof}

Let us next see how this structure allows us to easily dispense with case~\ref{allC3s}, and therefore with the case $\zeta(Q)=0$.
In fact for all of the cases we use the following claim.
\begin{claim}\label{clm:tree}
If there is a matching $M$ of edges from the same strong clique and a tree $T$ of the graph that is vertex-disjoint from $M$ such that some endpoint of every edge of $M$ is adjacent to a vertex in $T$, then there is a clique minor of order $|M|+1$.
\end{claim}\label{clm:clique}
\begin{claimproof}
Since they are part of a strong clique, contracting each of the edges of $M$ and then contracting $T$ yields a clique minor.
\end{claimproof}

Suppose now that we are in case~\ref{allC3s} and so $t=k-1$. Given ${\mathscr P}$, let us first consider the rooted triangles consecutively. For each $1\le i < p$, note by Claim~\ref{clm:roots} that $\rho_i$ and $\rho_{i+1}$ satisfy statement~\ref{complete}, so let $a_i,b_i$ in $\rho_i$ and $c_{i+1},d_{i+1}$ in $\rho_{i+1}$ denote the four corresponding vertices of the clique as in the statement of~\ref{complete}. By the pigeonhole principle we may assume by symmetry that $a_i = c_i$ for all $1<i<p$; thus, $P_{p}=a_1a_2\dots a_{p-1}c_p$ is a path in $H$. For each $1\le i<p$, we form a tree in the subgraph of $H$ induced by the vertices of all the triangles in $S_i$ as follows. We let $h(\rho_i) = a_i$ and from $\rho_i$ we explore the spanning arborescence $\Arb_i$ of $S_i$ (say, by depth-first search), and when we explore a directed edge $(\tau,\tau')$, where $\tau$ denotes the already explored vertex, we let $h(\tau')$ be any neighbour of $h(\tau)$ in the triangle $\tau'$ (which is guaranteed to exist by statement~\ref{dominate}). The set $\cup_{\tau\in S_i}\{h(\tau)\}$ induces a tree $T_i$ in $H$ that includes exactly one vertex from every triangle in $S_i$ and is rooted, say, at $a_i$. In the same way, we construct a tree $T_p$ in $H$ that includes exactly one vertex from every triangle in $S_p$ and is rooted at $c_p$. (If $p=1$, then $c_p$ is an arbitrary vertex of $\rho_p$.) Let $T$ be the tree in $H$ formed by appending the trees $T_1,\dots,T_p$ to the path $P_{p}$. This tree includes exactly one vertex from every triangle in ${\mathscr T}$. For each triangle in ${\mathscr T}$, add the edge that is not intersected by $T$ to the matching $M$ of $Q$. Thus $|M|=t=k-1$ and $T$ and $M$ satisfy the hypothesis of Claim~\ref{clm:tree}, so $H$ contains a clique minor of order $k$, a contradiction.

The cases~\ref{oneC2}--\ref{twoC5} are handled nearly the same way with the addition of one more structural observation.

\begin{claim}\label{clm:twoedges}
Let $\tau$ be a triangle and ${\cal D}$ some nontrivial cycle vertex-disjoint from $\tau$. If $\tau$ and ${\cal D}$ are in the same strong clique, then there must be at least two vertices of $\tau$ with a neighbour in ${\cal D}$.
\end{claim}

\begin{claimproof}
If not, let $u,v$ be distinct vertices of $\tau$ that do not have a neighbour in ${\cal D}$. This contradicts that the edge $uv$ and ${\cal D}$ are in the same strong clique.
\end{claimproof}

Assume we are in case~\ref{oneC2} or~\ref{oneC5} and assume that ${\cal D}_1$ is $K_2$ or the cycle of length five.
Now note that either $a_1$ or $b_1$ can play the role of the root of the tree $T$ constructed in case~\ref{allC3s}. By the pigeonhole principle, Claim~\ref{clm:twoedges} allows us to assume by symmetry that $a_1$ has a neighbour $u$ in ${\cal D}_1$. We let $T$ and $M$ be as constructed from ${\mathscr P}$ as in case~\ref{allC3s}.
In case~\ref{oneC2}, we add the edge of ${\cal D}_1$ to $M$.
In case~\ref{oneC5}, we append the edge $a_1u$ to $T$ and add the two independent edges of ${\cal D}_1-\{u\}$ to $M$.
Note that if $p=1$, we instead take $a_1=c_p$ in the above determinations.
In either case, $|M|=k-1$ and $T$ and $M$ satisfy the hypothesis of Claim~\ref{clm:tree}, which leads to a clique minor of order $k$.

Assume we are in one of cases~\ref{twoC2}--\ref{twoC5} and assume that ${\cal D}_1$ and ${\cal D}_2$ are the two non-triangles.
Now note that $c_p$ and $d_p$ are symmetric in the construction of the tree $T$ in case~\ref{allC3s}. As in the last two cases, we may assume that $a_1$ has a neighbour $u$ in ${\cal D}_1$. In the same way, we may assume that $c_p$ has a neighbour $v$ in ${\cal D}_2$. We let $T$ and $M$ be as constructed from ${\mathscr P}$ as in case~\ref{allC3s}.
In case~\ref{twoC2}, we add the two edges of ${\cal D}_1$ and ${\cal D}_2$ to $M$.
In case~\ref{oneC2oneC5}, we append the edge $c_pv$ to $T$ and add the edge of ${\cal D}_1$ and the two independent edges of ${\cal D}_2-\{v\}$ to $M$.
In case~\ref{twoC5}, we append the edges $a_1u$ and $c_pv$ to $T$ and add the four independent edges of ${\cal D}_1-\{u\}$ and ${\cal D}_2-\{v\}$ to $M$.
Note that if $p=1$, we may instead assume in the above determinations that $a_1=c_p$ is the common vertex of $\rho_1$ that has a neighbour in both ${\cal D}_1$ and ${\cal D}_2$.
In all cases, $|M|=k-1$ and $T$ and $M$ satisfy the hypothesis of Claim~\ref{clm:tree}, which leads to a clique minor of order $k$.

The proof of case~\ref{oneC7} is the most involved and for this we require the following structural observation.

\begin{claim}\label{clm:oneC3toC7}
Let $\tau$ be a triangle and ${\cal D}$ a cycle of length seven that is vertex-disjoint from $\tau$. Suppose $\tau$ and ${\cal D}$ are in the same strong clique and that there are two vertices $w,x$ of $\tau$ which each have no two neighbours that are at distance exactly three on ${\cal D}$.
Then the neighbours of $w$ in ${\cal D}$ form an interval of three consecutive vertices on ${\cal D}$ possibly less the middle vertex of the interval, the same is true for $x$, and the intervals of ${\cal D}$ corresponding to $w$ and $x$ are vertex-disjoint (and therefore adjacent in ${\cal D}$).
\end{claim}

\begin{claimproof}
If the conclusion of the claim does not hold, then there is an edge of ${\cal D}$ that has no edge from its endpoints to $w$ or $x$. This contradicts that such an edge and $wx$ are in the same strong clique.
\end{claimproof}

\noindent
Note that from the conclusion of Claim~\ref{clm:oneC3toC7}, we may also conclude that every triangle $\tau$ has at least one vertex with two neighbours at distance exactly three on the seven-cycle ${\cal D}$.

Assume we are in case~\ref{oneC7} and assume that ${\cal D}_1$ is the cycle of length seven.
If $p=1$, then the vertex $c_p=c_1$ in the construction of $T$ in case~\ref{allC3s} is arbitrarily chosen in $\tau_1$, and thus by Claim~\ref{clm:oneC3toC7} and up to symmetry we may assume that $c_p$ has two neighbours $u,v$ in ${\cal D}_1$ that are at distance exactly three on ${\cal D}_1$.
If $p>1$, then just as in cases~\ref{twoC2}--\ref{twoC5} we can note that the roles of $a_1$, $b_1$, $c_p$, and $d_p$ are symmetric in the construction of the tree $T$ in case~\ref{allC3s}.
If any of $a_1$, $b_1$, $c_p$, $d_p$ is adjacent to two neighbours $u,v$ at distance exactly three on ${\cal D}_1$, then by symmetry we may conclude that it is $c_p$ and conclude exactly as in the case $p=1$.
Therefore we can assume otherwise, so the two pairs $w=a_1$, $x=b_1$ and $w=c_p$, $x=d_p$ satisfy the conclusion of Claim~\ref{clm:oneC3toC7}.
From this structure we may assume without loss of generality that there are two vertices $u,v$ at distance exactly three on ${\cal D}_1$ such 
that $a_1v$ and $c_pu$ are edges of $H$.
In either of the cases $p=1$ or $p>1$, we let $T$ and $M$ be as constructed from ${\mathscr P}$ as in case~\ref{allC3s}.
Then we append the edge $c_pu$ to $T$ and add the three independent edges of ${\cal D}_1-\{u\}$ to $M$.
Then $|M|=k-1$ and $T$ and $M$ satisfy the hypothesis of Claim~\ref{clm:tree}, which leads to a clique minor of order $k$.
This concludes the proof.
\end{proof}

\section{A conditional fractional bound}\label{sec:fractional reduced}

Based on Lemma~\ref{lem:reduction:weighted}, we next present a seemingly innocent problem. 

\begin{conjecture}\label{conj:fractional reduced}
  Let $k\ge 4$.
  Let $G=(V,E)$ be a $K_k$-minor-free multigraph and $A\subseteq E$
  be a vertex-disjoint union of odd cycles and double edges in $G$.
  Then the fractional chromatic number of the subgraph of
  $L(G)^2$ induced by $A$ satisfies
  $\chi_f(L(G)^2[A]) \le 3(k-2)$.
\end{conjecture}
  Note that the bound in Conjecture~\ref{conj:fractional reduced} corresponds to
  the bound in Conjecture~\ref{conj:main} instanced with
  the maximum degree of the multigraph $(V,A)$, that is, $\Delta=2$.
  The following statement applied with $H$ as any $K_k$-minor-free graph
  and $\lambda=\frac{3}{2}(k-2)$ shows that
  Conjecture~\ref{conj:fractional reduced} is equivalent to the fractional
  relaxation of Conjecture~\ref{conj:main}.
  \begin{prop}\label{prop:fractional reduced}
    Let $H$ be a (simple) graph and let $\lambda$ be a real number.
    The following two statements are equivalent.
    \renewcommand{\labelenumi}{$(\roman{enumi})$}
    \begin{enumerate}
    \item\label{it:fractional conjecture}
      For every multigraph~$G$ with underlying simple graph $H$,
      $\chi_{2,f}'(G)\le \Delta(G) \cdot \lambda$.
    \item\label{it:fractional reduced}
      For every multigraph~$G$ with underlying simple graph $H$
      and every set of edges $A\subseteq E(G)$ that is a
      vertex-disjoint union of odd cycles and double edges in $G$,
      $\chi_{f}(L(G)^2[A]) \le 2\lambda$.
    \end{enumerate}
  \end{prop}
  \begin{proof}
    The proof relies on Lemma~\ref{lem:reduction:weighted} and
    the following correspondence.
    Given a multigraph $G$ with underlying simple graph $H$,
    the graph $L(G)^2$ has a fractional $r$-coloring
    $c:{\cal I}(L(G)^2)\to\mathbb{R}_{0}^{+}$
    if and only if there is a function
    $\widetilde{c}:{\cal I}(L(H)^2)\to\mathbb{R}_{0}^{+}$
    of the same total sum $r$, such that $\sum_{I\ni e}\widetilde{c}(I)$
    is the multiplicity of $e$ in $G$ for every $e\in E(H)$.
    
    We show first that~\ref{it:fractional reduced} implies~\ref{it:fractional conjecture}.
    Fix a multigraph~$G$ with underlying simple graph $H$.
    Let $w:E(H)\to\mathbb{N}$ be the function that assigns to each edge
    of~$H$ its multiplicity in~$G$.
    Let $w=\sum_{i=1}^pw_i$ be the decomposition provided by
    Lemma~\ref{lem:reduction:weighted} upon the input of $H$ and $w$.
    For every $i\in\{1,\dots,p\}$,
    consider the multigraph
    $G_i$ with underlying simple graph $H$, where each edge~$e\in E(H)$
    has multiplicity $\max(\frac{2}{\Delta(w_i)}w_i(e), 1)$.
    Property~\ref{itm:red4} of Lemma~\ref{lem:reduction:weighted}
    ensures that $\frac{2}{\Delta(w_i)}w_i(e)$ is
    an element of $\{0,1,2\}$.    
    Let $A_i$ be the subset of edges of $G_i$ containing exactly
    $\frac{2}{\Delta(w_i)}w_i(e)$ copies of the edge~$e$ for every $e\in E(H)$.
    By Properties~\ref{itm:red3} and~\ref{itm:red4} of Lemma~\ref{lem:reduction:weighted},
    $A_i$ is a vertex-disjoint union
    of odd cycles and double edges of $G$.
    As a consequence,~\ref{it:fractional reduced} of the current proposition
    applies to $A_i$ and
    yields
    $\chi_{f}(L(G_i)^2[A_i])\le 2\lambda$.
    Let $\widetilde{c}_i:{\cal I}(L(H)^2)\to{\mathbb R}^+$ be a corresponding
    function satisfying
    $\sum_{I\in{\cal I}(L(H)^2)}\widetilde{c}_i(I)=2\lambda$
    and such that $\sum_{I\ni e}\widetilde{c}_i(I)$ is the number of copies
    of $e$ in~$A_i$ for every~$e\in E(H)$.
    Now set $\widetilde{c}=\sum_{i=1}^p\frac{\Delta(w_i)}{2}\widetilde{c}_i$.
    We claim that $\widetilde{c}$ corresponds
    to a fractional $\lambda\Delta(G)$-coloring of $L(G)^2$.
    To see this, it suffices first to note that
    the total weight of $\widetilde{c}$
    is
    \[
      \sum_{I\in{\cal I}(L(H)^2)}\widetilde{c}(I)
      =\sum_{i=1}^p\frac{\Delta(w_i)}{2}\sum_{I\in{\cal I}(L(H)^2)}\widetilde{c}_i(I)
      =\lambda\sum_{i=1}^p\Delta(w_i)
      =\lambda\Delta(w)=\lambda\Delta(G),
    \]
    and
    that the coverage of each edge $e$ of $H$
    is
    \[
      \sum_{I\ni e}\widetilde{c}(I)=
      \sum_{i=1}^p\frac{\Delta(w_i)}{2}\sum_{I\ni e}\widetilde{c}_i(I)
      =\sum_{i=1}^pw_i(e)=w(e),
    \]
    which is exactly the multiplicity of~$e$ in~$G$.
    This proves that $\chi_{2,f}'(G)\le \lambda \Delta(G)$.

    In the other direction,
    assume~\ref{it:fractional conjecture}
    and let us prove~\ref{it:fractional reduced}.
    Let $G$ be a multigraph with underlying simple graph $H$ and let $A\subseteq E(G)$ be as in the statement.
    For a positive integer $D$,
    consider the multigraph $G_D$ obtained from $G$ by replacing each edge
    of $A$ by $D$ parallel edges.
    Applying~\ref{it:fractional conjecture} to $G_D$,
    we obtain a fractional coloring $c_D$ of $L(G_D)^2$
    with total weight
    at most $\lambda\Delta(G_D)\le \lambda(2D + \Delta(G))$.
    This implies that there exists a function $\widetilde{c}:{\cal I}(L(H)^2)\to\mathbb{R}_{0}^{+}$ of total sum  $\sum_{I\in{\cal I}(L(H)^2)}\widetilde{c}(I)$ at most $\lambda(2D + \Delta(G))$,
     such that $\sum_{I\ni e}\widetilde{c}(I)$ is at least the multiplicity of $e$ in $G_D$, for every $e\in E(H)$.
    We obtain that the normalised function $\widetilde{\widetilde{c}} := \frac{1}{D}\widetilde{c}_D$
    satisfies
    \[
      \sum_{I\in{\cal I}(L(H)^2)}\widetilde{\widetilde{c}}(I) \leq \frac{\lambda}{D}\cdot(2D+\Delta(G))
      =2\lambda + \frac{\lambda\Delta(G)}{D}
    \]
    and that
    $\sum_{I\ni e}\widetilde{\widetilde{c}}(I)$ is at least the multiplicity of
    $e$ in $A$ for every $e\in E(H)$. 
    Thus
    $\chi_f(L(G)^2[A])\le 2\lambda + \frac{\lambda\Delta(G)}{D}$
    and taking $D$ arbitrarily large gives the result.
  \end{proof}

\section{$K_4$-minor-free multigraphs}\label{sec:K4}

\begin{proof}[Proof of Theorem~\ref{thm:K4}]
  We prove by induction the following slightly stronger result.
  Given a $K_4$-minor-free multigraph~$G=(V,E)$ and a set of edges~$A\subseteq E$,
  the subgraph of~$L(G)^2$ induced by~$A$ has a proper $3\Delta_A$-vertex-colouring, where $\Delta_A$ denotes the maximum degree of $G[A]$.
  Theorem~\ref{thm:K4} therefore corresponds to the case $A=E$.

  For a multigraph~$G$, we write~$V_{\ge 2}(G)$ for the set of vertices of~$G$ with
  at least two neighbours.
  We prove the statement above by induction first on the cardinality
  of~$V_{\ge 2}(G)$, and second on the total number of vertices of~$G$.

  Note first that we may assume that~$G$ has no isolated vertices.
  
  If~$V_{\ge 2}(G)$ is empty, then $G$ is a union of vertex-disjoint multiedges,
  so the graph~$L(G)^2[A]$ is colourable with~$\Delta_A$
  colours.

  Assume now that~$V_{\ge 2}(G)$ is nonempty.
  It is known that every (nonempty) $K_4$-minor free simple graph
  has a vertex of degree at most~$2$~\cite{Dirac64}.
  Applying this result to the underlying simple graph of
  the induced submultigraph~$G[V_{\ge 2}(G)]$ yields
  a vertex~$v\in V_{\ge 2}(G)$
  with at most~$2$ neighbours in $V_{\ge 2}(G)$.
  We now consider two cases.

  First, assume that $v$ has a neighbour $w$ in $V(G)\setminus V_{\ge 2}(G)$
  (so the only neighbour of~$w$ is~$v$).
  Let~$A_w$ be the set of those edges in~$A$ that join~$v$ and~$w$.
  In this case, the degree in $L(G)^2[A]$
  of every edge $e\in A_w$ is at most~$3\Delta_A-1$.
  Indeed, $e'\in A$ is a neighbour of~$e$ in~$L(G)^2[A]$
  only if $e'$ is an edge of~$G$ incident to~$v$
  --- which concerns at most $\Delta_A-1$ other edges --- or incident to a
  neighbour $u\in V_{\ge 2}(G)$ of~$v$ (in~$G$)
  --- at most~$\Delta_A$ edges for each.
  Since~$v$ has at most two neighbours in~$V_{\ge 2}(G)$,
  this gives at most~$2\Delta_A + (\Delta_A-1)=3\Delta_A-1$ edges in total.
  It remains to apply induction on~$G':=G\setminus\{w\}$ and~$A':=A\setminus A_w$.
  This is possible because
  $|V_{\ge 2}(G')|\le |V_{\ge 2}(G)|$ and~$|V(G')|<|V(G)|$.
  It follows that $L(G)^2[A']$ has a proper colouring
  with at most $3\Delta_{A'}\le 3\Delta_{A}$ colours.
  Noting that moreover~$L(G')^2[A]=L(G)^2[A\setminus A_w]$
  and using the previous degree bound,
  this colouring can be extended to a $3\Delta_A$-colouring of~$L(G)^2$.

  Second, assume that~$v$ has no neighbour outside of~$V_{\ge 2}(G)$.
  Since~$v\in V_{\ge 2}(G)$, the vertex $v$ then has exactly two neighbours
  $u_1, u_2\in V_{\ge 2}(G)$.
  Let us consider the multigraph $G'$ constructed from~$G$ by ``splitting'' $v$
  into two new vertices~$v_1$ and~$v_2$ as follows.
  This construction is depicted on Figure~\ref{fig:proof for k4}.
  Start with~$G\setminus\{v\}$ and add an edge $u_1u_2$
  if no such edge is already present.
  Then for~$i\in\{1,2\}$, add the vertex~$v_i$, and for each edge $u_iv$ in~$G$,
  add an edge~$u_iv_i$. Let~$A'$ be the set $A$ where every edge~$u_iv$ is
  changed into the corresponding edge~$u_iv_i$.
  Observe that~$L(G)^2[A]$ is a subgraph of $L(G')^2[A']$,
  so every proper colouring of~$L(G')^2[A']$ is a proper colouring of~$L(G)^2[A]$.
  Also note that $G'$ is $K_4$-minor-free because $v_1$ and~$v_2$ each have only one
  neighbour in $G'$ and $G'\setminus\{v_1,v_2\}$ is a minor of~$G$.
  As $V_{\ge 2}(G')=V_{\ge 2}(G)\setminus\{v\}$,
  induction applied to~$G'$ and~$A'$ gives
  a proper colouring of~$L(G')^2[A']$ with $3\Delta_{A'}$ colours.
  It remains to note that~$\Delta_{A'}\le\Delta_{A}$ to conclude the proof.
\end{proof}

\noindent
Wang, Wang and Wang~\cite{WWW18} proved a similar bound for simple $K_4$-free graphs.

\begin{figure}
  \centering
  \begin{tabular}{c c c}
    \begin{tikzpicture}[scale=0.5]
      \tikzstyle{v}=[black,circle,draw,fill=black,scale=0.4];
      \node[v,label=-90:$v$] (v) at (0, -2) {};
      \foreach\i/\x in {1/-3,2/3} {
        \foreach \a in {45, 75, ..., 135} {
          \draw[xshift=\x cm] (\a:0.8) -- (0,0); }
        \node[v,label={180*\i}:$u_\i$] (u\i) at (\x, 0) {};
        \foreach\a in {-21, -7, ..., 21}{
          \draw[bend left=\a] (u\i) to (v);}
      }
    \end{tikzpicture} &  &
    \begin{tikzpicture}[scale=0.5]
      \tikzstyle{v}=[black,circle,draw,fill=black,scale=0.4];
      \foreach\i/\x in {1/-3,2/3} {
        \foreach \a in {45, 75, ..., 135} {
          \draw[xshift=\x cm] (\a:0.8) -- (0,0); }
        \node[v,label={180*\i}:$u_\i$] (u\i) at (\x, 0) {};
        \node[v,label={-90}:$v_\i$] (v\i) at (2/3*\x, -3) {};
        \foreach\a in {-21, -7, ..., 21}{
          \draw[bend left=\a] (u\i) to (v\i);}
      }
      \draw (u1) -- (u2);
    \end{tikzpicture}\\
    $G$ & & $G'$
  \end{tabular}
  \caption{One of the transformations used in the proof of Theorem~\ref{thm:K4}.}
  \label{fig:proof for k4}
\end{figure}
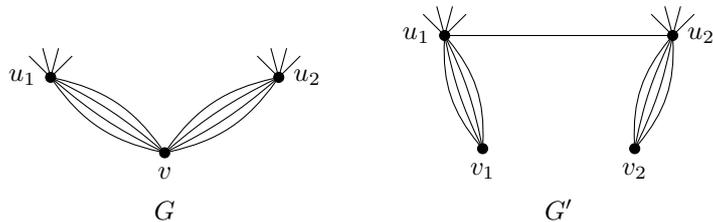

\section{Multigraphs and matchings that are strong cliques}\label{sec:isclique}

Our main objective here is to treat the $K_k$-minor-free multigraphs of edge-diameter $2$. Relatedly, instead of $K_k$-minor-freeness we consider what happens when we impose the slightly weaker condition that there is no matching of $k$ edges that are pairwise connected by an edge. These considerations are focused mainly on the strong clique number.

First, let us describe the supplemental multigraph construction that asymptotically matches Theorem~\ref{thm:isclique}. For each $k\ge 5$ and each $\Delta \ge k-2$, we construct a multigraph $S_{k,\Delta}$ as follows.

The vertex set of $S_{k,\Delta}$ is the disjoint union $A \cup B \cup \left\{c,d\right\}$, where $A= \left\{a_1,\ldots, a_{k-1} \right\}$ and  $B= \left\{b_2,\ldots, b_{k-1} \right\}$. For the edges, we have the following. The set $A$ induces a clique minus the edge $a_1a_2$. For each $2 \le i \le k-1$, there is a multiedge $a_ib_i$ of multiplicity $\Delta-(k-1)$. The vertex set $\left\{a_1,c,d\right\}$ induces a triangle which is almost a Shannon triangle, in the following sense: the edges $a_1c,a_1d, cd$ are multiedges of multiplicities $ \lceil(\Delta-(k-3))/2 \rceil, \lfloor (\Delta-(k-3))/2 \rfloor$ and $\lceil (\Delta-3)/2 \rceil$, respectively. 

Finally, $c$ is joined by a simple edge to each of $a_2,a_3,\ldots, a_{1+\lfloor (k-1)/2 \rfloor}$ and similarly, $d$ is joined by a simple edge to $a_2$ and to  $a_{2+\lfloor (k-1)/2 \rfloor}, \ldots, a_{k-1}$. Note that $A$ equals the open neighbourhood $N(\left\{c,d\right\})$, and that $N(c) \cap N(d)= \left\{a_1,a_2\right\}$. See Figure~\ref{fig:infernal}.

  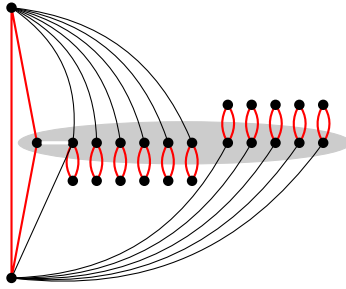
\begin{figure}[!ht]
  \begin{center}
 \begin{tikzpicture}[scale=0.112]
       \tikzstyle{v}=[black,circle,draw,fill=black,scale=0.3];
       \tikzstyle{edge}=[];
       \tikzstyle{noedge}=[very thick, white];
       \tikzstyle{clique}=[thick, red];

       \draw[gray!40,fill] (17.3,0) ellipse (19.5 and 2.5);

       \foreach\r in {2,...,7} {
         \draw[clique,bend left]
         (2.8*\r-2.8+1.4,0)node[v](A\r){} to (2.8*\r-2.8+1.4,-4.5)node[v](B\r){} to (A\r);
       }
              \foreach\r in {8,...,12} {
         \draw[clique,bend left]
         (2.8*\r-2.8+2.8,0)node[v](A\r){} to (2.8*\r-2.8+2.8,4.5)node[v](B\r){} to (A\r);
       }
       
       \draw[clique]
       (0,0)node[v](A1){}--(-3,16)node[v](C){}--(-3,-16)node[v](D){}--(A1);
       \draw[noedge](A1)--(A2);

       \foreach\r in {2,...,7} {
         \draw[edge,bend right]
         (A\r) to (C);
       }
       \foreach\r in {8,...,12} {
         \draw[edge,bend left]
         (A\r) to (D);
       }
       \draw[edge](D) to (A2);

     \end{tikzpicture}
  \end{center}
  \caption{A schematic of the graph $S_{k,\Delta}$. The edges coloured red are of nontrivial multiplicity, indeed of roughly $\Delta/2$. The white line denotes a non-edge.}\label{fig:infernal}
\end{figure}

Observe that the square of the line graph of $S_{k,\Delta}$ is indeed a clique. Thus the strong clique number of $S_{k,\Delta}$ is simply the cardinality of its edge set. The graph induced by $A$ has $\binom{k-1}{2}-1$ edges, the edges that have one endpoint in $A \backslash \left\{a_1\right\}$ and the other endpoint outside $A$ contribute $(k-2)(\Delta-(k-2))+1$ and the triangle induced by $\left\{a_1,c,d\right\}$ contributes $ \lceil \frac{3}{2}(\Delta+1) \rceil -k$ edges. Therefore 
\begin{align*}
\omega'_2(S_{k,\Delta})
&= (k-2)(\Delta-(k-2)) + \left\lceil \frac{3}{2}(\Delta+1) \right\rceil -k + \binom{k-1}{2}\\
&= \left\lceil \left(k-\frac{1}{2}\right)\Delta - \binom{k-1}{2} -\frac{1}{2} \right\rceil.
\end{align*}

\begin{prop}\label{prop:infernal}
The multigraph $S_{k,\Delta}$ is $K_k$-minor-free.
\end{prop}
\begin{proof}
We may replace every multiedge with a simple edge. Furthermore, since each vertex of $B$ has only one neighbour, we may contract the edge $a_ib_i$, for each $2 \le i \le k-1$.  If the resulting graph $T$ is $K_k$-minor-free then so is $S_{k,\Delta}$.  Now, because $T$ has $k+1$ vertices, it suffices to show that $K_{k}$ cannot be obtained from $T$ by contracting a single edge. To conclude that this is indeed not possible, it suffices to find three vertex-disjoint non-edges in $T$.  Because $k\ge 5$, it follows that $a_1a_2$, $ca_{k-1}$ and $da_3$ satisfy this requirement.
\end{proof}

Next we prove Theorem~\ref{thm:isclique} using the Tutte--Berge formula, extending an argument from~\cite{LoKa19}.
Given a multigraph $G$, let $\mu(G)$ denote the size of the largest matching of $G$, and let $o(G)$ denote the number of components of $G$ that have an odd number of vertices. By the Tutte--Berge formula~\cite{TutteBerge1958}, it holds for any multigraph $G=(V,E)$ that
$\mu(G)= \frac{1}{2} \min_{U \subseteq V} \left( |U| -o(G-U) + |V| \right)$.

We actually prove the following slightly stronger form of Theorem~\ref{thm:isclique}, in which the condition of $K_k$-minor-freeness is relaxed to the condition that $G$ does not contain a $k$-edge matching. 

\begin{thm}\label{thm:iscliquestronger}
Let $k\ge 5$.
For any multigraph $G$ of maximum degree $\Delta$ and matching number less than $k$ such that every two of its edges are incident or joined by an edge, 
the strong clique number of $G$ satisfies
$\omega_2'(G)  \le (k - \frac12 ) \Delta.$ 
\end{thm}

\begin{proof}
First note that for any $U\subseteq V(G)$, the graph $G-U$ has at most one component with an edge. This is because edges in different components of $G-U$ cannot be incident or joined by an edge in $G$. 

Let $U \subseteq V(G)$ be such that $\mu(G)=\frac{1}{2} \left( |U| -o(G-U) + |V(G)| \right)$.  If $G-U$ has no edges, then $o(G-U)=|V(G)|-|U|$, so that $\mu(G)= |U|$ and $|E(G)|\le \sum_{u \in U} \deg(u) \le \Delta \cdot |U| \le \Delta \cdot (k-1)$. 
We may thus assume that $G-U$ has a component with an edge. Let $A$ denote the set of vertices in this unique nontrivial component. If $|A|$ is even then $|V(G)|=|U| + o(G-U) +|A|$. If $|A|$ is odd then $|V(G)|=|U| + o(G-U) +|A|-1$. Therefore
\(
\mu(G) \ge |U| + (|A|-1)/2.
\)
It follows that 
\begin{align*}
|E(G)|
&\le \sum_{u \in U} \deg(u) +  |E(G[A])|
\le \Delta \cdot ( |U| + |A|/2)
 \le \Delta \cdot ( \mu(G) + 1/2)&\\
  &\le  \Delta \cdot ( k - 1/2). &\qedhere
\end{align*}
\end{proof}

It is natural to wonder how the conclusion of Theorem~\ref{thm:main} is affected by relaxing $K_k$-minor-freeness in a similar way as in Theorem~\ref{thm:iscliquestronger}. In this case we can easily obtain an upper bound that is only an additive factor $\frac32\Delta$ larger than the bound in Theorem~\ref{thm:main}. 

\begin{prop}\label{prop:edgepacking}
For any multigraph $G$ of maximum degree $\Delta$ that does not contain $k$ vertex-disjoint edges every two of which are joined by an edge, the strong clique number of $G$ satisfies $\omega'_2(G) \le \frac{3}{2} (k-1) \Delta$.
\end{prop}
\begin{proof}
Let $H$ be a submultigraph of $G$ such that $E(H)$ forms a maximum clique in $L(G)^2$, so in particular $|E(H)|=\omega'_2(G)$. By Shannon's Theorem, the edges of $H$ can be properly coloured with $\tfrac{3}{2}\Delta(H)$ colours. Therefore $H$ has a matching $M$ of size $|M| \ge |E(H)|/(\frac32 \Delta(H))$. Since every two edges of $H$ are joined by an edge in $G$, we have $|M|\le k-1$.
\end{proof}

\noindent
It remains unclear to what extent Proposition~\ref{prop:edgepacking} is sharp. For $k\ge 5$, it seems plausible that Theorem~\ref{thm:main} also holds under the relaxed condition of Proposition~\ref{prop:edgepacking}. On the other hand, such a strengthening does not hold if $k=4$, since then (for $\Delta$ large enough) there exists a multigraph as in Proposition~\ref{prop:edgepacking} such that $\omega_2'(G)= \frac{7}{2}\Delta + O(k) > 3\Delta$, namely the graph obtained from $G_{k,\Delta}$ in Section~\ref{sec:diabolical} by deleting the two vertices $a$ and $a'$.

Next we remark on a nonasymptotic strengthening of Theorem~\ref{thm:simple}, also in the more general setting of graphs without $k$ vertex-disjoint edges every two of which are joined by an edge. First note that by Vizing's Theorem every graph $G$ of maximum degree $\Delta$ has matching number at least $|E(G)| / (\Delta+1)$. This can be slightly improved according to the following result of Chv\'atal and Hanson~\cite{CH76}, which is sharp throughout the range of $\Delta$ and the matching number $\mu$, cf.~\cite{BK09}.

\begin{thm}[\cite{CH76}]\label{thm:matching}
For any graph $G=(V,E)$ with matching number $\mu$ and maximum degree $\Delta$,
\[|E| \le \mu \Delta + \left\lfloor \frac{\mu}{ \left\lceil  \frac{\Delta}{2}\right\rceil} \right\rfloor \cdot \left\lfloor \frac{\Delta}{2}\right\rfloor.\]
\end{thm}

\noindent
As a corollary, we obtain the following bound for the strong clique number.

\begin{cor}\label{cor:simplegraph}
Let $G$ be a graph of maximum degree $\Delta$ that does not contain $k$ vertex-disjoint edges every two of which are joined by an edge. Then the strong clique number of $G$ satisfies 
\[
\omega'_2(G) \le  (k-1)\Delta + \left\lfloor \frac{k-1}{ \left\lceil  \frac{\Delta}{2}\right\rceil} \right\rfloor \cdot \left\lfloor \frac{\Delta}{2}\right\rfloor.
\]
\end{cor}
\begin{proof}
Let $H$ be a subgraph of $G$ such that $E(H)$ is a maximum clique in $L(G)^2$. Since every two vertex-disjoint edges of $H$ are joined by an edge in $G$, the matching number of $H$ is at most $k-1$. Applying Theorem~\ref{thm:matching} to $H$ yields the desired upper bound on $\omega'_2(G)=|E(H)|$.
\end{proof}

\noindent
Corollary~\ref{cor:simplegraph} is sharp for the odd cliques, which correspond to the case $\Delta=2k-2$, and for the complete bipartite graph $K_{k-1,\Delta}$ if $\Delta \ge 2k-1$.

For small values of $\Delta $ compared to $\mu$, the known extremal graphs for Theorem~\ref{thm:matching} are disconnected; they consist of a disjoint union of singletons and connected graphs $G_i$ with $(\mu(G_i)+1/2)\Delta(G_i)$ edges~\cite{CH76}.
For that reason, and keeping in mind Theorem~\ref{thm:iscliquestronger}, its sharpness for the blown-up $5$-cycle, as well as the fact that $(k-1)(\Delta+1)=(k-\frac12)\Delta$ when $\Delta=2k-2$, we believe that the following improvement upon Corollary~\ref{cor:simplegraph} could be true.
\begin{conjecture}
Let $G$ be a graph of maximum degree $\Delta$ that does not contain $k$ vertex-disjoint edges every two of which are joined by an edge. Then the strong clique number of $G$ satisfies
\[\omega'_2(G) \le
\begin{cases}
(k-\frac12)\Delta & \mbox{ if }\Delta \le 2k-2;\\
(k-1) \Delta & \mbox{ if } \Delta \ge 2k-1.
\end{cases}
\]
\end{conjecture}

\noindent
We comment that the corresponding conclusion for strong edge-colouring fails by the existence of graphs of girth $6$ having strong chromatic index $\Omega(\Delta^2/\log \Delta)$.

\subsection*{Acknowledgement}

We are grateful to the anonymous referees for their careful reading and comments, which led to improvements in the presentation of this work.

\bibliographystyle{abbrv}
\bibliography{stronghadwiger}

\end{document}